\documentclass[12pt,leqno]{article}
\usepackage{amsmath,amssymb,amsthm,amsfonts,mathrsfs,amscd,environ,bbm}
\usepackage{latexsym,enumerate,color,geometry}
\usepackage{fancyhdr,xcolor}
\usepackage{verbatim}
 \NewEnviron{ews}{%
\begin{equation}\begin{split}
  \BODY
\end{split}\end{equation}
}

\NewEnviron{ews*}{%
\begin{equation*}\begin{split}
  \BODY
\end{split}\end{equation*}
}

\def\beg{\begin}
\def\bequ{\begin{equation}}
\def\enqu{\end{equation}}
\def\bes{\begin{split}}
\def\ens{\end{split}}
\def\bews{\begin{ews}}
\def\beqn{\begin{eqnarray}}
\def\enqn{\end{eqnarray}}
\def\beq*{\begin{equation*}}
\def\enq*{\end{equation*}}
\def\bqn*{\begin{eqnarray*}}
\def\eqn*{\end{eqnarray*}}
\def\bary{\begin{array}}
\def\eary{\end{array}}
\def\bpma{\begin{pmatrix}}
\def\epma{\end{pmatrix}}
\def\bvma{\begin{Vmatrix}}
\def\evma{\end{Vmatrix}}

 \newtheorem{thm}{Theorem}[section]
 \newtheorem{lem}[thm]{Lemma}

 \numberwithin{equation}{section}

\def\al{\alpha}
\def\be{\beta}
\def\ga{\gamma}
\def\de{\delta}
\def\ep{\epsilon}

\def\ze{\zeta}
\def\et{\eta}
\def\th{\theta}

\def\la{\lambda}
\def\rh{\rho}

\def\ta{\tau}

\def\De{\Delta}

\def\Ph{\Phi}
\def\Ps{\Psi}
\def\Om{\Omega}

\def\Q{\mathbb Q}
\def\R{\mathbb R}
\def\P{\mathbb P}
\def\E{\mathbb E}
\def\H{\mathbb H}
\def\N{\mathbb N}
\def\V{\mathbb V}

\def\sF{\mathscr F}
\def\sD{\mathscr D}

\def\sB{\mathscr B}

\def\sL{\mathscr L}

\def\d{\mathrm{d}}

\def\ff{\frac}
\def\ra{\rightarrow}

\def\<{\langle}
\def\>{\rangle}
\def\sq{\sqrt}
\def\tld{\tilde}
\def\we{\wedge}
\def\1{\mathbbm{1}}

\allowdisplaybreaks

\title{{\bf Strong Feller Property and Irreducibility for Non-Linear Monotone SPDEs }
}

\author{
{\bf Shao-Qin Zhang }\\
\footnotesize{School of Statistics and Mathematics, Central University of Finance and Economics, Beijing 100081, China}\\
\footnotesize{Email: zhangsq@cufe.edu.cn}\\
}

\begin{document}
\maketitle

\begin{abstract}
Strong Feller property and irreducibility are study for a class of non-linear monotone stochastic partial differential equations with multiplicative noise. H\"older continuity of the associated Markov semigroups are discussed in some special cases. The main results are applied to several examples such as stochastic porous media equations, stochastic fast diffusion equations.
\end{abstract}\noindent

\noindent


\vskip 2cm

\section{Introduction}
In the present paper, we concern the strong Feller property and irreducibility of the transition semigroups corresponding to a class of non-linear monotone stochastic partial differential equations(SPDEs) with multiplicative noise. Both of them are important in the study of ergodicity of stochastic systems, for example, it is well known that strong Feller property and irreducibility ensure the uniqueness of invariant measure of the transition semigroup, see \cite{PZ95,DPZ1996}. Indeed, there have been abundant literatures contributing to the study of these two properties for semilinear SPDEs with either additive or multiplicative noise. It is a pity that we can not to list them all here, and only mention monograph \cite{DPZ1996}  for a systematic introduction on ergodicity for semilinear SPDEs and \cite{PZ95} for the fist result on this topic about semilinear SPDEs with multiplicative noise. On the other hand, the non-linear monotone SPDEs basing on the variational approach have been intensively investigated recently. The research of this type SPDEs can be dated back to Pardoux \cite{Pardoux}. Further development is given by Krylov and Rozovoskii \cite{KrylovR}.  For the existence and uniqueness of strong and weak solutions, we refer to \cite{DPR04,BBDR,DRRW06,RenRW,LiuR10,Liu13,LiuR13}. Besides this, we shall mention that \cite{BDR,RW13,Gess13} for extinction problem of solutions, and \cite{Wang2007,LiuW08,Liu09,WBook13'}  for Harnack inequalities and related topics, and \cite{Wang14b} for the first result on the gradient estimate for this type of equations. Though the dimensional free Harnack inequalities introduced by \cite{Wang97} are powerful tool to study strong Feller property and irreducibility, the existing results on this type inequalities for such equations are mainly for additive noise.

Before further discussion, let us recall the framework of this type of equations briefly. Let $(\H,\<\cdot,\cdot\>, |\cdot|)$ be a real separable Hilbert space. Let $\V$ be a reflective Banach space continuously and densely embedded into $\H$ and $\V^*$ be the dual space with respect to $\H$. We denote the dualization between $\V^*$ and $\V$ by $_{\V^*}\<\cdot,\cdot\>_{\V}$. Then $_{\V^*}\<u,v\>_{\V}=\<u,v\>$ for $u\in \H$, $v\in \V$. $(\V,\H,\V^*)$ is called a Gelfand triple. Let $\sL(\H)$($\sL_2(\H)$) be the set of all bounded(resp. Hilbert-Schmidt) operators on $\H$. We denote the operator norm and Hilbert-Schmidt norm by $|\cdot|$ and $|\cdot|_2$ respectively. Consider the following non-Linear SPDEs on $\H$
\bequ\label{equ_main}
\d X(t)=A(X(t))\d t+B(X(t))\d W(t),
\enqu
where $\{W(t)\}_{t\geq 0}$ is cylindrical Brownian motion on $\H$ with respect to a complete filtered probability spaces $\Big(\Om,\sF,\P,\{\sF_t\}_{t\geq 0}\Big)$. The coefficients $$A: \V\ra \V^*,\ B: \V\ra \sL_2(\H)$$
are measurable and such that the following assumptions hold
\beg{enumerate}[({H}1)]
\item (Hemicontinuity) For all $v_1,v_2,v\in \V$, $\R\ni s\ra _{\V^*}\<A(v_1+sv_2),v\>_\V$ is continuous.\label{H_hemi}
\item (Monotonicity) There exists $K_1\in\R$ such that for all $v_1,v_2\in \V$,\label{H_mono}
$$2_{\V^*}\<A(v_1)-A(v_2),v_1-v_2\>_\V+|B(v_1)-B(v_2)|_2^2\leq K_1|v_1-v_2|^2.$$
\item (Coercivity) There exist $r>0$ and $c_1,c_3\in \R,c_2>0$ such that for all $v\in \V$,\label{H_coe}
$$2_{\V^*}\<A(v),v\>_\V+|B(v)|_2^2\leq c_1-c_2|v|_\V^{r+1}+c_3|v|^2.$$
\item (Growth) There exists $c_3,c_4\in\R$ such that\label{H_grow}
$$|_{\V^*}\<A(u),v\>_\V|\leq c_4+c_5(|u|_\V^{r}+|v|_\V^{r+1}+|u|^2+|v|^2).$$
\end{enumerate}
\beg{defn}
A continuous $\H$-valued adapted process $X$ is called a solution to (\ref{equ_main}), if
$$\E\int_0^T\Big(|X(t)|_\V^{r+1}+|X(t)|^2\Big)\d t<\infty, T>0,$$
and $\P$-a.s.
$$X(t)=X(0)+\int_0^tA(X(s))\d s+\int_0^tB(X(s))\d s, t>0.$$
\end{defn}
According to \cite{PRo2007,LiuR10}, for $X(0)\in L^2(\Om\ra \H,\P)$, (H\ref{H_hemi})-(H\ref{H_grow}) ensure that (\ref{equ_main}) has a unique pathwise continuous solution. Denote $\sB_b(\H)$ as all bounded Borel measurable functions on $\H$. Let $X^x(t)$ be the solution of $\eqref{equ_main}$ starting from $x$, and $$P_tf(x)=\E f(X^x(t)),\ f\in \sB_b(\H),\ x\in \H.$$
Then $P_t$ is the associated Markov semigroup of (\ref{equ_main}).

The main aim of this paper is to study the strong Feller property and irreducibility of $P_t$ via coupling by change of measure. The coupling used to prove strong Feller property here was first introduced in \cite{Wang2007}. In \cite{ZhangX09}, the strong Feller property and irreducibility for stochastic differential equations with non-Lipschitz and monotone coefficients are concerned using the coupling  as well as the approximative controllability method, moreover, the exponential ergodicity is obtained there. We shall follow the line of \cite{Wang2007} and \cite{ZhangX09}. Here, the drift term $A:\V\ra \V^*$ is usually singular, and can not be taken in the sense of usual nonlinear function in finite dimensional case. Moreover, though we shall assume that the noise term $B$ is non-degenerate,  it is Hilbert-Schmidt operator, that means it is weaker than the usual semilinear SPDEs or SDEs cases where $\sup_{x\in \H}|B^{-1}(x)|<\infty$, see \cite{PZ95,DPZ1996,ZhangX09}. So another detailed discussion for this type of equations is needed. As in \cite{ZhangX09}, or \cite{RoWang2010,Wang2011,ZhangSQ,WBook13'},  we shall introduce the following non-degenerate condition for the SPDEs discussed here:
\beg{itemize}
\item ({\bf{Non-degenerate condition}}) There exist a positive self-adjoint operator $B_0$, and a measurable function $\rh:\V\ra (0,\infty)$ which satisfies
    $$\inf_{v\in\V,|v|\leq R}\rh(v)>0,~R>0,$$
    such that
\bequ\label{nondenerate_B}
B(v)B^*(v)\geq \rh(v)B_0^2,\ v\in \V.
\enqu
\end{itemize}
Define the intrinsic norm induced by $B_0$ as follow
$$
|u|_{B_0}=\beg{cases}
|B_0^{-1}u|,~&u\in B_0(\H),\\
~\infty,~&u\notin B_0(\H).
\end{cases}
$$
 At last, we shall mention a recent result in \cite{Wang14b} which provided the gradient/H\"older estimates for this type of SPDEs for the first time. That is a quite strong result. But there, the noise term is split into a multiplicative part and an additive part  independently. It is different from us.

The following two theorems are devoted to the strong Feller property.
\beg{thm}\label{thm:sFeller}
Let $r\geq 1$. Assume that
\beg{enumerate}[(1)]
\item~(H\ref{H_hemi}), (H\ref{H_coe})-(H\ref{H_grow}) and the non-degenerate condition (\ref{nondenerate_B}) hold, $|\cdot|_{B_0}:\V\ra [0,\infty)$ is measurable, and for all $n\in \N$, there exists $C_n\in \R$ such that
\bequ\label{locally_B_Lip}
|B(v_1)-B(v_2)|_2\leq C_n|v_1-v_2|,\ v_1,v_2\in \V, |v_1|\vee|v_2|\leq n,
\enqu
\item~there exist $\th\in [2,+\infty)\cap (r-1,+\infty)$ and $K_1\geq 0,~\de_1>0$ such that
\beg{ews}\label{HofA}
&2_{\V^*}\<A(v_1)-A(v_2),v_1-v_2\>_\V+|B(v_1)-B(v_2)|_2^2\\
&\quad\leq -\de_1|v_1-v_2|_{B_0}^\th|v_1-v_2|^{r+1-\th}+K_1|v_1-v_2|^2,~~v_1,v_2\in \V.
\end{ews}
\end{enumerate}
Then
$P_t$ is strong Feller.
\end{thm}
\beg{thm}\label{thm:sFeller'}
Let $0<r<1$. Assume that
\beg{enumerate}[(1)]
\item~(H\ref{H_hemi}), (H\ref{H_grow}) and the non-degenerate condition (\ref{nondenerate_B}) hold, $|\cdot|_{B_0}:\V\ra [0,\infty)$ is measurable, and $B$ is locally Lipschitz in the sense of \eqref{locally_B_Lip},
\item~there exists $\th\geq \ff 4 {r+1}$, $K_1\geq 0,~\de_1>0$ and measurable function $h:\V\ra [0,+\infty)$ such that
\beg{ews}\label{HofA'}
&2_{\V^*}\<A(v_1)-A(v_2),v_1-v_2\>_\V+|B(v_1)-B(v_2)|_2^2\\
&\quad\leq -\ff {\de_1|v_1-v_2|_{B_0}^\th} {|v_1-v_2|^{\th-2}\Big(h(v_1)\vee h(v_2)\Big)^{1-r}}+K_1|v_1-v_2|^2,~~v_1,v_2\in \V,
\end{ews}
\item there exist $c_1,c_3\in \R,~c_2>0$ such that
\bequ\label{H_coe'}
2_{\V^*}\<A(v),v\>_{\V}+|B(v)|^2_2\leq c_1-c_2h^{1+r}(v)+c_3|u|^2,~~v\in \V.
\enqu
\end{enumerate}
Then $P_t$ is strong Feller.
\end{thm}
Typical examples for Theorem \ref{thm:sFeller} are stochastic porous media equations and stochastic p-Laplace equations. Theorem \ref{thm:sFeller'} can be applied to stochastic fast diffusion equations. Conditions \eqref{HofA}, \eqref{HofA'} and \eqref{H_coe'} have been used in the previous works on studying Harnack inequalities, see \cite{Wang2007,LiuW08,Liu09,WBook13'}. For technical reason, we still need local Lipschitz condition \eqref{locally_B_Lip}. If $\inf_{v\in \V} \rh(v)>0$, by some global conditions, we can get that $P_t f$ is H\"older continuous on $\H$ for all $f\in\sB_b(\H)$, see Lemma \ref{lmm:sFeller} and \ref{lmm:sFeller'}.

The next theorem is for irreducibility. $P_t$ is said to be irreducibility at $t>0$ if for arbitrary non empty open set $S\subset \H$ and all $x\in \H$,
$$P_t(x,S)\equiv P_t1_{S}(x)>0.$$
We generalize a result in \cite{BaDa} to multiplicative noise. For more about using approximative controllability method to prove the irreducibility, we would like to refer \cite{DPZ1996}.

\beg{thm}\label{thm:irreducibility}
Assume that (H\ref{H_hemi})-(H\ref{H_grow}) hold, the non-degenerate condition (\ref{nondenerate_B}) holds with
\bequ\label{rh_irr}
\inf_{v\in\V}\rh(v)>0,
\enqu
and $|\cdot|_{B_0}:\V\ra[0,+\infty)$ is measurable, and for all $n\in \N$, there exists $C_n$ such that
\bequ\label{con_B_2}
|B(u)|_2\leq C_n,\ u\in \V,~|u|\leq n.
\enqu
Then
$P_t$ is irreducible.
\end{thm}

Though we assume that $B$ is locally bounded on $\H$, we should point out that it follows from (H\ref{H_coe}) and (H\ref{H_grow}) that there exists constant $\tld c>0$ such that
\bequ\label{bounded_B}
|B(u)|_2^2\leq \tld c\Big(1+|u|_\V^{r+1}+|u|^2\Big).
\enqu

The rest parts of the paper are organized as follows. In the following section, we give the proofs of the main results. Simple applications and examples are provided in Section 3.

\section{Proof of main results}
Fix $T>0$. We shall concern strong Feller property and irreducibility for $P_T$. Firstly, we shall prove the strong Feller property for $P_T$ under some global conditions. Throughout this section, (H\ref{H_hemi})-(H\ref{H_grow}) are assumed to be hold.
\begin{lem}\label{lmm:sFeller}
Assume that (\ref{nondenerate_B}) and (\ref{HofA}) hold for some self-adjoint Hilbert-Schmidt operator $B_0$ and $\rh$ with $\inf_{v\in \V}\rh(v)=\ff 1 {c_0}>0$. \\
{\rm{(1)}}~If there exists $K_2\geq0$ such that
\bequ\label{strong_cond_B}
|(B(u)-B(v))^*(u-v)|\leq K_2\Big(|u-v|^2\we|u-v|\Big),~~u,v\in \V.
\enqu
Then $P_T f$ is $\ff {\th-r+1} {2\th}$-H\"older continuous on $\H$ for all $f\in \sB_b(\H)$.\\
{\rm{(2)}}~If $\th >r$ and there exists $K_2\geq 0$ such that
\bequ\label{weak_cond_B}
|(B(u)-B(v))^*(u-v)|\leq K_2|u-v|,~~u,v\in \V.
\enqu
Then $P_Tf$ is $\ff {\th-r+1} {2\th}$-H\"older continuous on $\H$ for all $f\in \sB_b(\H)$.
\end{lem}
\begin{proof} We prove {\rm{(1)}} firstly. Let $\ep\in (0,1)$ such that $0\vee (r-1)<\th(1-\ep)<(2r)\we(r+1)$. Consider the following coupling equations
\beg{ews}\label{coupling}
\d X(t)&=A(X(t))\d t+B(X(t))\d W(t),\ X(0)=x\\
\d Y(t)&=A(Y(t))\d t+B(Y(t))\d W(t)+|x-y|^{\al}\ff {X(t)-Y(t)} {|X(t)-Y(t)|^{\ep}}\d t,\ Y(0)=y,
\end{ews}
where $0<\al<\ep$. By It\^o's formula,
\beg{ews}\label{ItoforX_Y}
&\d |X(t)-Y(t)|^2\\
&\quad=2\<A(X(t))-A(Y(t)),X(t)-Y(t)\>\d t+|B(X(t))-B(Y(t))|_2^2\d t\\
&\quad\quad +2\<(B(X(t))-B(Y(t)))\d W(t),X(t)-Y(t)\>-2|x-y|^{\al}|X(t)-Y(t)|^{2-\ep}\d t.
\end{ews}
According to (\ref{HofA}) and (\ref{ItoforX_Y}), we also have
\beg{ews}
&\d |X(t)-Y(t)|^2\\
&\quad\leq -\de_1|X(t)-Y(t)|_{B_0}^\th|X(t)-Y(t)|^{r+1-\th}\d t+K_1|X(t)-Y(t)|^2\d t\\
&\quad\quad +2\<(B(X(t))-B(Y(t)))\d W(t),X(t)-Y(t)\>-2|x-y|^{\al}|X(t)-Y(t)|^{2-\ep}\d t.
\end{ews}
Let $2\ga=r+1-\th(1-\ep)$. Then $\ga>0$, $r+1-\th-2\ga=-\th\ep$, moreover
$$\ga=\ff {r+1-\th(1-\ep)} 2< \ff {r+1-(r-1)\vee0} 2\leq 1.$$
So
\beg{ews}\label{Ito_ga}
&\d |X(t)-Y(t)|^{2-2\ga}\\
&\quad\leq-(1-\ga)\de_1\ff {|X(t)-Y(t)|_{B_0}^{\th}} {|X(t)-Y(t)|^{\th\ep}}\d t+(1-\ga)K_1|X(t)-Y(t)|^{2-2\ga}\d t\\
&\quad\quad+2(1-\ga)\Big\<\ff{(B(X(t))-B(Y(t)))^*(X(t)-Y(t))} {|X(t)-Y(t)|^{2\ga}},\d W(t)\Big\>.
\end{ews}
Let
$$\ta_n=\inf\{t>0~|~|X(t)-Y(t)|\geq \ff 1 n\},~~\ta=\inf\{t>0~|~|X(t)-Y(t)|=0\}.$$
Then
\bequ\label{int_t}
\E\int_0^{T\we\ta}\ff {|X(t)-Y(t)|_{B_0}^{\th}} {|X(t)-Y(t)|^{\th\ep}}\d t\leq \ff {e^{(1-\ga)K_1T}} {(1-\ga)\de_1} |x-y|^{2-2\ga}.
\enqu
We can define
$$\ze_n=\inf\{s>0~|~\int_0^s\ff {|X(t)-Y(t)|_{B_0}^{\th}} {|X(t)-Y(t)|^{\th\ep}}\d t\geq n\}.$$
Due to \eqref{int_t}, $\lim_{n\ra \infty}\ze_n\geq \ta$. Let
$$\eta_n=\ze_n\we\ta_n\we T.$$
Then $\lim_{n\ra\infty}\eta_n=\ta\we T$. Let $\tld B^{-1}(y)=B^*(y)(B(y)B^*(y))^{-1}$,
\bequ\label{tldW}
\tld W(t)=W(t) +\int_0^{t\we\ta}|x-y|^\al\ff {\tld B^{-1}(Y(s))(X(s)-Y(s))} {|X(s)-Y(s)|^{\ep}}\d s,
\enqu
and
\beg{ews}\label{R_T}
R_t&=\exp\Big\{|x-y|^\al\int_0^{t\we\ta}\Big\<\ff {\tld B^{-1}(Y(s))(X(s)-Y(s))} {|X(s)-Y(s)|^{\ep}},\d W(s)\Big\>\\
&\qquad-\ff {|x-y|^{2\al}} 2\int_0^{t\we\ta}\ff {|\tld B^{-1}(Y(s))(X(s)-Y(s))|^2} {|X(s)-Y(s)|^{2\ep}}\d s\Big\}.
\end{ews}
By Novikov's condition, $\{W(s)\}_{0\leq s\leq \eta_n\we t}$ is cylindrical Brownian motion on $\H$ under probability measure $R_{t\we\et_n}\P$. Next, we shall prove that $\{W(t)\}_{0\leq t\leq T}$ is cylindrical Brownian motion under $R_T\P$.

In the following, we denote $\Q_{n,t}(\Q)$ as the probability measure $R_{t\we\eta_n}\P$(resp. $R_T\P$) and $\E_{\Q_{n,t}}$($\E_{\Q}$) the expectation with respect to $R_{\eta_n,t}\P$(resp. $R_T\P$). According to \eqref{Ito_ga} and \eqref{tldW},
\beg{ews*}
&\d |X(t)-Y(t)|^{2-2\ga}\\
&\quad\leq-(1-\ga)\de_1\ff {|X(t)-Y(t)|_{B_0}^{\th}} {|X(t)-Y(t)|^{\th\ep}}\d t+(1-\ga)K_1|X(t)-Y(t)|^{2-2\ga}\d t\\
&\quad\quad-2(1-\ga)|x-y|^\al\Big\<\ff{(B(X(t))-B(Y(t)))\tld B^{-1}(Y(t))(X(t)-Y(t))} {|X(t)-Y(t)|^{2\ga}},\ff {X(t)-Y(t)} {|X(t)-Y(t)|^{2\ga}}\Big\>\d t\\
&\quad\quad+2(1-\ga)\Big\<\ff{(B(X(t))-B(Y(t)))^*(X(t)-Y(t))} {|X(t)-Y(t)|^{2\ga}},\d \tld W(t)\Big\>.
\end{ews*}
By H\"older inequality, there is $C>0$ such that
\beg{ews*}
&\Big|-2(1-\ga)|x-y|^\al\Big\<\ff{(B(X(t))-B(Y(t)))\tld B^{-1}(Y(t))(X(t)-Y(t))} {|X(t)-Y(t)|^{2\ga}},\ff {X(t)-Y(t)} {|X(t)-Y(t)|^{\ep}}\Big\>\Big|\\
&\quad\leq 2(1-\ga)c_0|x-y|^{\al}K_2\Big(|X(t)-Y(t)|^{1-2\ga}\we|X(t)-Y(t)|^{2-2\ga}\Big)\ff {|X(t)-Y(t)|_{B_0}} {|X(t)-Y(t)|^{\ep}}\\
&\quad\leq C(1-\ga)^{\ff {\th} {\th-1}}|x-y|^{\ff {\al\th} {\th-1}}\Big(|X(t)-Y(t)|^{1-2\ga}\we|X(t)-Y(t)|^{2-2\ga}\Big)^{\ff {\th} {\th-1}}\\
&\quad\quad+\ff {(1-\ga)\de_1} {2}\ff {|X(t)-Y(t)|^\th_{B_0}} {|X(t)-Y(t)|^{\th\ep}}.
\end{ews*}
Moreover, since $(1-2\ga)\ff {\th} {\th-1}\leq 2-2\ga$, we get that
$$\Big(|X(t)-Y(t)|^{1-2\ga}\we|X(t)-Y(t)|^{2-2\ga}\Big)^{\ff {\th} {\th-1}}\leq |X(t)-Y(t)|^{2-2\ga}.$$
Hence
\beg{ews*}
&\d |X(t)-Y(t)|^{2-2\ga}\\
&\quad\leq -\ff {\de_1} {2}\ff {|X(t)-Y(t)|^\th_{B_0}} {|X(t)-Y(t)|^{\th\ep}}\d t+2(1-\ga)\Big\<\ff{(B(X(t))-B(Y(t)))^*(X(t)-Y(t))} {|X(t)-Y(t)|^{2\ga}},\d \tld W(t)\Big\>\\
&\quad\quad+\Big(C|x-y|^{\ff {\al\th} {\th-1}}(1-\ga)^{\ff 1 {\th-1}}+K_1\Big)(1-\ga)|X(t)-Y(t)|^{2-2\ga}\d t.
\end{ews*}
Taking expectation with respect to $\Q_{n,s}$, we have that
$$\E_{\Q_{n,s}}\int_0^{s\we\et_n}\ff {|X(t)-Y(t)|^\th_{B_0}} {|X(t)-Y(t)|^{\th\ep}}\d t\leq \ff {2e^{(C|x-y|^{\ff {\al\th} {\th-1}}(1-\ga)^{\ff 1 {\th-1}}+K_1)(1-\ga)s}} {(1-\ga)\de_1}|x-y|^{2-2\ga},~~0\leq s\leq T.$$
Consequently,
\bequ\label{RlogR}
\sup_{s\in[0,T],n\geq 1}\E R_{\eta_n\we s}\log R_{\eta_n\we s}\leq \ff {2T^{\ff {\th} {\th-2}}e^{\Big(C|x-y|^{\ff {\al\th} {\th-1}}(1-\ga)^{\ff 1 {\th-1}}+K_1\Big)\ff {2(1-\ga)T} {\th}}} {((1-\ga)\de_1)^{2/\th}}|x-y|^{\ff {4(1-\ga)} {\th}+2\al},
\enqu
due to H\"older inequality. So, $\{R_{\eta_n\we s}\}_{s\in[0,T],n}$ is uniformly integrable. Consequently, $\{R_s\}_{s\in[0,T]}$ is uniformly integrable martingale.  Then, by Girsanov's theorem, $\{\tld W(t)\}_{0\leq t\leq T}$ is cylindrical Wiener process under the probability measure $R_T\P$. Moreover, $Y(t)$ is a weak solution of the following equation
\bequ
\d Y(t)=A(Y(t))\d t+B(Y(t))\d \tld W(t),\ Y(0)=y.
\enqu
Thus, we get that
\beg{ews}
&|P_Tf(x)-P_Tf(y)|=|\E f(X(T))-\E R_Tf(Y(t))|\\
&=|\E f(Y(T))(1-R_T)|+|\E(f(X(T))-f(Y(T)))\mathbbm{1}_{\{\ta\geq T\}}|\\
&\leq |f|_\infty|\Big(\E|1-R_T|+\P(\ta\geq T)\Big),~f\in\sB_b(H).
\end{ews}
The last task is to give estimations of the two term in the last inequality.

By \eqref{RlogR},~\eqref{int_t} and
$$|1-e^x|\leq  xe^x+2|x|\we 1,~~x\in \R,$$
there is $C>0$ such that
\beg{ews}
\E|1-R_T|&\leq \E R_T\log R_T+2\E|\log R_T|\\
&\leq C\Big(|x-y|^{\ff {2(1-\ga)} {\th}+\al}+|x-y|^{\ff {4(1-\ga)} {\th}+2\al}\Big).
\end{ews}
On the other hand, according to (\ref{ItoforX_Y}), we have
\beg{align*}
\d |X(t)-Y(t)|^{\ep}&=\ep/2 |X(t)-Y(t)|^{\ep-2}\d |X(t)-Y(t)|^2\\
&\quad +\ff 1 2 \ep(\ep/2-1)|X(t)-Y(t)|^{\ep-4}|(B(X(t))-B(Y(t)))^*(X(t)-Y(t))|^2\d t\\
&\leq \ff {\ep K_1} {2}|X(t)-Y(t)|^{\ep}\d t-|x-y|^\al\d t\\
&\quad+\ep\Big\<(B(X(t))-B(Y(t)))\d W(t),\ff {X(t)-Y(t)} {|X(t)-Y(t)|^{2-\ep}}\Big\>.
\end{align*}
Then
\beq*
\E |X(s\we\ta)-Y(s\we\ta)|^\ep\leq |x-y|^\ep + \ff {\ep K_1} {2}\E\int_0^{s\we\ta}|X(t)-Y(t)|^\ep\d t-|x-y|^\al\E(s\we\ta).
\enq*
By Gronwall's inequality, we have that
\beg{ews*}
\E |X(s\we\ta)-Y(s\we\ta)|^\ep&\leq |x-y|^\ep e^{\ff {\ep K_1s} {2}},\\
\int_0^s \E|X(t\we\ta)-Y(t\we\ta)|^\ep\d t&\leq \ff {2|x-y|^\ep (e^{\ff {\ep K_1s} {2}}-1)}{\ep K_1s},~s\in[0,T].
\end{ews*}
Hence
$$\E(s\we\ta)\leq |x-y|^{\ep-\al}\ff {e^{\ff {\ep K_1s} {2}}-1+s} s,~s\in[0,T],$$
and
$$\P(\ta\geq T)\leq \ff {\E(\ta\we T)} {T}\leq |x-y|^{\ep-\al}\ff {e^{\ff {\ep K_1T} {2}}-1+T^2} {T^2},~s\in[0,T].$$
Therefore, there exists constant $C$ depending on $|f|_\infty, \de_1,|B_0|,K_1,K_2,T,\ep,r,\th$ such that
$$|P_Tf(x)-P_Tf(y)|\leq C|x-y|^\be,$$
with
$$\be=\Big(\ff {2(1-\ga)} {\th}+\al\Big)\we (\ep-\al).$$
Let $\al=\ff {\ep} 2-\ff {1-\ga} {\th}$. Then $\be=\ff {\th-r+1} {2\th}$. The proof of {\rm{(1)}} is completed.

For {\rm{(2)}}. Since $\th>r$, we can choose $\ep=1-\ff r {\th}$. Repeating the argument used in {\rm{(1)}} with $\ga=1/2$ and \eqref{weak_cond_B} instead of \eqref{strong_cond_B}, we obtain the claim in {\rm{(2)}}.
\end{proof}

\bigskip

\noindent{\bf{\emph{Proof of Theorem \ref{thm:sFeller}:~}}} We truncate $B$ as follow
$$B_R(v)=
\begin{cases}
B(v),~~~~~|v|\leq R,\\
B(\ff {Rv} {|v|}),~~~|v|> R,
\end{cases}$$
where $R>0$.  By \eqref{locally_B_Lip}, it is clear that $B_R$ is a bounded Lipschitz map, hence \eqref{strong_cond_B} and \eqref{weak_cond_B} in Lemma \ref{lmm:sFeller} hold for $B_R$ with some constant $C_R>0$. Moreover, we have the following non-degenerate condition:
$$B_R(v)B^*_R(v)\geq \rh_R B_0^2=(\sq\rh_R B_0)^{2},$$
where $\rh_R=\inf_{|v|\leq R}\rh(v)$. Let
$$\ta_R^x=\inf\{t>0~|~|X^x(t)|\geq R\},~~~\ta_R^y=\inf\{t>0~|~|X^y(t)|\geq R\},$$
and $$X^{R,x}(t),~X^{R,y}(t)$$
be the solution of the following equation
\bequ\label{equ_tranc}
\d X^R(t)=A(X^R(t))\d t+B_R(X^R(t))\d W(t)
\enqu
starting form $x,~y$ respectively. Let $P_T^Rf(x)=\E f(X^{R,x}(T))$. By the uniqueness of (\ref{equ_tranc}), $X^x(T)=X^{R,x}(T)$ and $X^y(T)=X^{R,y}(T)$ for $T< \ta_R^x\we\ta_R^y$. So
\beg{align}\label{inequ:sFeller}
&|P_Tf(x)-P_Tf(y)|\nonumber\\
&\quad\leq|\E(f(X^x(T))-f(X^y(T)))1_{\{\ta_R^x\we\ta_R^y> T\}}|+|f|_\infty\Big(\P\Big(\ta_R^x\leq T\Big)+\P\Big(\ta_R^y\leq T\Big)\Big)\nonumber\\
&\quad= |\E(f(X^{R,x}(T))-f(X^{R,y}(T)))1_{\{\ta_R^x\we\ta_R^y> T\}}|+|f|_\infty\Big(\P\Big(\ta_R^x\leq T\Big)+\P\Big(\ta_R^y\leq T\Big)\Big)\\
&\quad\leq |P_T^Rf(x)-P_T^Rf(y)|+2|f|_\infty\Big(\P\Big(\ta_R^x\leq T\Big)+\P\Big(\ta_R^y\leq T\Big)\Big).\nonumber
\end{align}
Applying Lemma \ref{lmm:sFeller} to (\ref{equ_tranc}), $P^R_T$ is strong Feller. On the other hand, by It\^o's formula and (H\ref{H_coe}), we have that
\beg{ews}\label{ItoforX}
\d |X^x(t)|^2\leq c_1-c_2|X^x(t)|_\V^{r+1}+c_3|X^x(t)|^2\d t+2\<B(X^x(t))\d W(t), X^x(t)\>,~x\in \H.
\end{ews}
According to Gronwall's inequality, we obtain that
\beg{ews}\label{inequ:X}
\E|X^x(t)|^2+\int_0^t\E|X^x(t)|_\V^{r+1}\d t\leq \Big(\ff {c_1t+|x|^2} {c_2\we 1}\Big)\Big(1+c_3e^{\ff {c_3t} {c_2\we 1}}\Big),~t\in [0,T],~x\in \H.
\end{ews}
Combining (\ref{ItoforX}), (\ref{inequ:X}) with B-D-G inequality and (H\ref{H_coe}), (H\ref{H_grow}), there exists $C>0$ which is independent of $T$ and $x$ such that
\bequ
\E\sup_{t\in[0,T]}|X^x(t)|^2\leq C\Big(1+e^{CT}\Big)\Big(T+|x|^2\Big),~x\in \H.
\enqu
Since
\beg{ews*}
&\P\Big(\ta_R^x\leq T\Big)\leq \P\Big(\sup_{t\in[0,T]}|X(t)|\geq R\Big)\\
&\leq \ff 1 {R^2}\E\sup_{t\in[0,T]}|X(t)|^2\leq \ff C {R^2}\Big(1+e^{CT}\Big)\Big(T+|x|^2\Big),~x\in \H,
\end{ews*}
we have
$$\sup_{z\in B(x,l)}\P\Big(\ta_R^z\leq T\Big)\leq \ff C {R^2}\Big(1+e^{CT}\Big)\Big(T+(|x|+l)^2\Big),~x\in \H,~l>0.$$
Letting $y\ra x$ in (\ref{inequ:sFeller}) first(assume that $|y-x|\leq 1$) and then $R\uparrow\infty$, we obtain the strong Feller property of $P_T$:
$$\lim_{y\ra x}|P_Tf(x)-P_Tf(y)|\leq 2|f|_\infty\lim_{R\uparrow \infty}\sup_{z\in B(x,1)}\P(\ta_T^z\leq T)=0.$$
\qed

\bigskip

To prove Theorem \ref{thm:sFeller'}, we only have to prove the following lemma.
\beg{lem}\label{lmm:sFeller'}
Assume that conditions of Theorem \ref{thm:sFeller'} hold with $\inf_{v\in\V}\rh(v)= \ff 1 {c_0}>0$ and the locally Lipschitz condition \eqref{locally_B_Lip} replaced by the boundedness of $B$: there exists $C>0$ such that
$$|B(v)|\leq C,~~v\in \V.$$
Then $P_T$ is locally $(\ff {2\th} {3\th+4}\we \ff 1 2)$-H\"older continuous.
\end{lem}
\beg{proof}
Let $\ep=\ff {\th} {\th+2}$. Consider the coupling equations as \eqref{coupling}. By \eqref{HofA'} and It\^o's formula, we have
\beg{ews}
\d |X(t)-Y(t)|^2&\leq -\ff {\de_2|X(t)-Y(t)|_{B_0}^\th} {|X(t)-Y(t)|^{\th-2}(h(X(t))\vee h(Y(t)))^{1-r}}\d t\\
&\quad+K_1|X(t)-Y(t)|^2\d t-2|x-y|^\al|X(t)-Y(t)|^{2-\ep}\d t\\
&\quad+2\Big\<(B(X(t))-B(Y(t)))\d W(t),X(t)-Y(t)\Big\>.
\end{ews}
Then
\begin{align}\label{Ito'}
&\d |X(t)-Y(t)|^{2\ep}\nonumber\\
&\quad\leq -\ff {\ep\de_2|X(t)-Y(t)|^\th_{B_0}} {|X(t)-Y(t)|^{2(1-\ep)+\th-2}(h(X(t))\vee h(Y(t)))^{1-r}}\d t\nonumber\\
&\quad\quad+\ep K_1|X(t)-Y(t)|^{2\ep}\d t+2\ep\Big\<(B(X(t))-B(Y(t)))\d W(t),\ff {X(t)-Y(t)} {|X(t)-Y(t)|^{2(1-\ep)}}\Big\>\\
&\quad\quad+\ep(\ep-1)|X(t)-Y(t)|^{2(\ep-2)}|(B(X(t))-B(Y(t)))^*(X(t)-Y(t))|^2\d t.\nonumber
\end{align}
Let
$$\ta_n=\inf\{t>0~|~|X(t)-Y(t)|\leq 1/n\},~~\ta=\inf\{t>0~|~|X(t)-Y(t)|=0\}.$$
It is clear that
\beg{ews}\label{zeta}
\E\int_0^{t\we\ta}\ff {|X(s)-Y(s)|^\th_{B_0}} {|X(s)-Y(s)|^{2(1-\ep)+\th-2}(h(X(s))\vee h(Y(s)))^{1-r}}\d s\leq \ff { e^{\ep K_1t}|x-y|^{2\ep}} {\ep\de_2},~~t>0.
\end{ews}
Let
$$\ze_n=\inf\Big\{t>0~|~\int_0^t\ff {|X(s)-Y(s)|_{B_0}^\th} {|X(s)-Y(s)|^{\th\ep}(h(X(s))\vee h(Y(s)))^{1-r}}\d s\geq n\Big\}.$$
By \eqref{zeta}, we have
$$\lim_{n\ra\infty}\ze_n\geq \ta, ~~\P\mbox{-a.s}.$$
On the other hand, by It\^o's formula and \eqref{H_coe'}, we get that
\beg{ews}
\d |X(t)|^2\leq c_1\d t-c_2h^{1+r}(X(t))\d t+c_3|X(t)|^2\d t+2\<B(X(t))\d W(t),X(t)\>
\end{ews}
So
$$\E\int_0^{t}h^{1+r}(X(s))\d s\leq \ff {e^{c_3T}} {c_2}\Big[\ff {c_1(1-e^{-c_3T})} {c_3}+|x|^2\Big],~~t\leq T.$$
According to It\^o's formula, it is easy to see that
$$\E\int_0^{t}h^{1+r}(Y(s))\d s<\infty,~~t>0.$$
Define
\bqn*
&&\xi_n=\inf\Big\{t>0~|~\int_0^t \Big[h^{1+r}(X(s))\vee h^{1+r}(Y(s))\Big]\d s\geq n\Big\},\\
&&\eta_n=\ta_n\we\ze_n\we\xi_n\we T.
\eqn*
Then $\lim_{n\ra\infty}\xi_n=\infty$ and $\lim_{n\ra\infty}\eta_n=T\we\ta.$
Let $\tld W(t)$ and $R_t$ defined as in \eqref{tldW} and \eqref{R_T}. Since
\beg{ews}\label{ForRlogR}
&\int_0^{t\we\et_n}\ff {|X(s)-Y(s)|_{B_0}^2} {|X(s)-Y(s)|^{2\ep}}\d s\\
&\quad\leq \Big(\int_0^{t\we\et_n}\ff {|X(s)-Y(s)|_{B_0}^\th} {|X(s)-Y(s)|^{\th\ep}(h(X(s))\vee h(Y(X(s))))^{1-r}}\d s\Big)^{\ff 2 {\th}}\\
&\quad\quad\times\Big(\int_0^{t\we\et_n}(h(X(s))\vee h(Y(s)))^{\ff {2(1-r)} {\th-2}}\d s\Big)^{\ff {\th-2} {\th}},
\end{ews}
$\{\tld W(s)\}_{0\leq s\leq \eta_n\we t}$ is cylindrical Brownian Motion under probability measure $\Q_{n,t}:=R_{\eta_n\we t}\P$ according to Noikov's condition, moreover
\bequ\label{Q_nY}
\E_{\Q_{n,s}}\int_0^{\eta_n\we s}h^{1+r}(Y(s))\d s\leq \ff {e^{c_3T}} {c_2}\Big[\ff {c_1(1-e^{-c_3T})} {c_3}+|y|^2\Big],~s\in[0,T].
\enqu

Next, we shall prove that $\{\tld W(t)\}_{0\leq t \leq T}$ is cylindrical Brownian Motion on $\H$. To this end, we shall prove that $\sup_{t\in[0,T]}\E(R_t\log R_t)$ is finite.
By It\^o's formula and \eqref{H_coe'},
\beg{ews}
\d |X(t)|^2&\leq c_1\d t-c_2h^{1+r}(X(t))\d t+c_3|X(t)|^2\d t+2\<B(X(t))\d W(t),X(t)\>\\
&=c_1\d t-c_2h^{1+r}(X(t))\d t+c_3|X(t)|^2\d t+2\<B(X(t))\d \tld W(t),X(t)\>\\
&\quad-\ff {2|x-y|^\al\Big\<B(X(t))\tld B^{-1}(Y(t))(X(t)-Y(t)),X(t)\Big\>} {|X(t)-Y(t)|^\ep}\d t.
\end{ews}
Since $\th  \geq \ff {4} {r+1}$, we have $\ff {1-r} {(1+r)\th} +\ff 1 2+\ff 1 {\th}\leq 1$. Then, by H\"older inequality, there exist nonnegative constants $\tld c_1,~\tld c_2,~\tld c_3$ such that
\beg{ews*}
&\ff {|x-y|^{\al}|X(t)-Y(t)|_{B_0}|X(t)|} {|X(t)-Y(t)|^\ep}\\
&\quad\leq \tld c_3+\tld c_1|x-y|^{2\al}|X(t)|^2+\ff {c_2} 2(h(X(t))\vee h(Y(t)))^{1+r}\\
&\quad\quad+ \ff {\tld c_2|X(t)-Y(t)|^{\th}_{B_0}} {|X(t)-Y(t)|^{\th\ep}(h(X(t))\vee h(Y(t)))^{1-r}}.
\end{ews*}
So, there exist $\hat c_1,~\hat c_2,~\hat c_3\in \R^+$ such that
\beg{ews*}
\d |X(t)|^2&\leq \hat c_1\d t-c_2h^{1+r}(X(t))\d t+\hat c_3(1+|x-y|^{2\al})|X(t)|^2\d t\\
&\quad +\ff {\hat c_2|X(t)-Y(t)|_{B_0}^\th} {|X(t)-Y(t)|^{\th\ep}(h(X(t))\vee h(Y(t)))^{1-r}}\d t\\
&\quad+\ff {c_2} 2(h^{1+r}(X(t))+h^{1+r}(Y(t)))\d t+2\<B(X(t))\d \tld W(t),X(t)\>.
\end{ews*}
Denote $\E_{\Q_{n,s}}$ as the expectation with respective to probability $\Q_{n,s}$. We obtain that
\beg{align*}
&\E_{\Q_{n,s}}\int_0^{\eta_n\we s}e^{-\hat c_3(1+|x-y|^{2\al})t}h^{1+r}(X(t))\d t\\
&\quad\leq \ff {2|x|^2} {c_2}+\ff {\hat c_1(1-e^{-{\hat c_3(1+|x-y|^{2\al})T}})} {c_2\hat c_3(1+|x-y|^{2\al})}+\E_{\Q_{n,s}}\int_0^{\eta_n\we s}e^{-\hat c_3(1+|x-y|^{2\al})t}h^{1+r}(Y(t))\d t \\
&\quad\quad+\ff {2\hat c_2} {c_2}\E_{\Q_{n,s}}\int_0^{\eta_n\we s}\ff {e^{-\hat c_3(1+|x-y|^{2\al})t}|X(t)-Y(t)|_{B_0}^\th} {|X(t)-Y(t)|^{\th\ep}(h(X(t))\vee h(Y(t)))^{1-r}}\d t\\
&\quad\leq \ff {2|x|^2} {c_2}+\ff {\hat c_1(1-e^{-{\hat c_3(1+|x-y|^{2\al})T}})} {c_2\hat c_3(1+|x-y|^{2\al})}+ \ff {e^{c_3T}} {c_2}\Big[\ff {c_1(1-e^{-c_3T})} {c_3}+|x|^2\Big]\\
&\quad\quad+\ff {2\hat c_2} {c_2}\E_{\Q_{n,s}}\int_0^{\eta_n\we s}\ff {e^{-\hat c_3(1+|x-y|^{2\al})t}|X(t)-Y(t)|_{B_0}^\th} {|X(t)-Y(t)|^{\th\ep}(h(X(t))\vee h(Y(t)))^{1-r}}\d t,~s\in[0,T],
\end{align*}
combining this with \eqref{Q_nY}, we get that
\beg{ews*}
&\E_{\Q_{n,s}}\int_0^{\eta_n\we s}(h(X(t))\vee h(Y(t)))^{1+r}\d t\\
&\quad\leq \E_{\Q_{n,s}}\int_0^{\eta_n\we s}h^{1+r}(X(t))\d t+\E_{\Q_{n,s}}\int_0^{\eta_n\we s}h^{1+r}(Y(t))\d t\\
&\quad\leq \ff {2e^{\hat c_3(1+|x-y|^{2\al})T}} {c_2}\Big\{|x|^2+\ff {\hat c_1(1-e^{-\hat c_3(1+|x-y|^{2\al})T})} {\hat c_3(1+|x-y|^{2\al})}\\
&\quad\quad+\ff {e^{c_3T}} 2\Big[\ff {c_1(1-e^{-c_3T})} {c_3}+|y|^2\Big]\Big\}+\ff {e^{c_3T}} {c_2}\Big[\ff {c_1(1-e^{-c_3T})} {c_3}+|y|^2\Big]\\
&\quad\quad+\ff {2\hat c_2} {c_2} e^{\hat c_3(1+|x-y|^{2\al})}\E_{\Q_n}\int_0^{\eta_n}\ff {|X(t)-Y(t)|_{B_0}^\th} {|X(t)-Y(t)|^{\th\ep}(h(X(t))\vee h(Y(t)))^{1-r}}\d t,~s\in[0,T].
\end{ews*}
By \eqref{Ito'}, we have
\begin{align*}
&\d |X(t)-Y(t)|^{2\ep}\\
&\quad\leq -\ff {\ep\de_2|X(t)-Y(t)|^\th_{B_0}} {|X(t)-Y(t)|^{\th-2\ep}(h(X(t))\vee h(Y(t)))^{1-r}}\d t+\ep K_1|X(t)-Y(t)|^{2\ep}\d t\\
&\quad\quad+2\ep\Big\<(B(X(t))-B(Y(t)))\d \tld W(t),\ff {X(t)-Y(t)} {|X(t)-Y(t)|^{2(1-\ep)}}\Big\>\\
&\quad\quad-\ff {2\ep|x-y|^\al\Big\<(B(X(t))-B(Y(t)))\tld B^{-1}(Y(t))(X(t)-Y(t)),X(t)-Y(t)\Big\>} {|X(t)-Y(t)|^{2-\ep}}.
\end{align*}
Since $B(\cdot)$ is bounded, we have that
\begin{align*}
&\Big|\ff {2\ep|x-y|^\al\Big\<(B(X(t))-B(Y(t)))\tld B^{-1}(Y(t))(X(t)-Y(t)),X(t)-Y(t)\Big\>} {|X(t)-Y(t)|^{2-\ep}}\Big|\\
&\quad\leq \ff {2c_0\ep|x-y|^\al|X(t)-Y(t)|_{B_0}|(B(X(t))-B(Y(t)))^*(X(t)-Y(t))|} {|X(t)-Y(t)|^{2-\ep}}\\
&\quad=\ff {2C\ep|x-y|^\al|X(t)-Y(t)|_{B_0}|X(t)-Y(t)|^{2\ep-1}(h(X(t))\vee h(Y(t)))^{\ff {1-r} {\th}}} {|X(t)-Y(t)|^{\th\ep}(h(X(t))\vee h(Y(t)))^{\ff {1-r} {\th}}}\\
&\quad\leq \ff {\de_2\ep|X(t)-Y(t)|_{B_0}^\th} {2|X(t)-Y(t)|^{\th\ep}(h(X(t))\vee h(Y(t)))^{1-r}}\\
&\quad\quad+\tld C|x-y|^{\ff {\al\th} {\th-1}}|X(t)-Y(t)|^{\ff {\th(\th-2)} {(\th+2)(\th-1)}}(h(X(t))\vee h(Y(t)))^{\ff {1-r} {\th-1}},
\end{align*}
for some $\tld C>0$. Since $\ff {\th(\th-2)} {2(\th+2)(\th-1)\ep}+\ff {1-r} {(\th-1)(r+1)}<1$, by H\"older inequality, we get that for all $l>0$, there is $c(l)>0$ such that
\beg{ews*}
\d |X(t)-Y(t)|^{2\ep}&\leq -\ff {\de_1\ep|X(t)-Y(t)|_{B_0}^\th} {2|X(t)-Y(t)|^{\th\ep}(h(X(t))\vee h(Y(t)))^{1-r}}\d t\\
&\quad+(\ep K_1+c(l))|X(t)-Y(t)|^{2\ep}\d t\\
&\quad+l|x-y|^{2\al}\Big[(h(X(t))\vee h(Y(t)))^{1+r}+1\Big]\d t\\
&\quad+2\ep\Big\<(B(X(t))-B(Y(t)))\d \tld W(t),\ff {X(t)-Y(t)} {|X(t)-Y(t)|^{2(1-\ep)}}\Big\>,~s\in[0,T].
\end{ews*}
Hence
\beg{ews*}
&\E_{\Q_{n,s}}\int_0^{\eta_n\we s}\ff {|X(t)-Y(t)|^{\th}_{B_0}} {|X(t)-Y(t)|^{\th\ep}(h(X(t))\vee h(Y(t)))^{1-r}}\d t\\
&\quad\leq\ff {2le^{(\ep K_1+c(l))T}} {\de_2\ep} |x-y|^{2\al}\Big[\E_{\Q_{n,s}}\int_0^{\eta_n\we s}(h(X(t))\vee h(Y(t))^{1+r}\d t+\ff {1-e^{-(\ep K_1+c(l))T}} {\ep K_1+c(l)}\Big]\\
&\quad\quad+\ff {2e^{(\ep K_1+c(l))T}} {\de_2\ep}|x-y|^{2\ep},~s\in[0,T].
\end{ews*}
If $|x-y|\leq \Big(\ff {\de_2\ep} {4le^{(\ep K_1+c(l))T}}\Big)^{\ff 1 {2\al}}$, then
\beg{ews*}
&\E_{\Q_n}\int_0^{\eta_n}\ff {|X(t)-Y(t)|^{\th}_{B_0}} {|X(t)-Y(t)|^{\th\ep}(h(X(t))\vee h(Y(t)))^{1-r}}\d t\\
&\quad\leq \ff {4} {\de_2\ep} \Big[e^{(\ep K_1+C(l))T}|x-y|^{2\ep}+|x-y|^{2\al}\ff {\de_2\ep(1-e^{-(\ep K_1+C(l))})} {\ep K_1+C(l)}\Big].
\end{ews*}
According to \eqref{ForRlogR} and Fatou lemma, we get that for $|x-y|\leq \Big(\ff {\de_2\ep} {4le^{(\ep K_1+c(l))T}}\Big)^{\ff 1 {2\al}}$
\begin{align*}
\sup_{s\in[0,T]}\E R_s\log R_s &=\ff {|x-y|^{2\al}} 2\E_{\Q}\int_0^{T\we\ta}\ff {|\tld B^{-1}(Y(t))(X(t)-Y(t))|^2} {|X(t)-Y(t)|^{2\ep}}\d t\\
&\leq \ff {c_0|x-y|^{2\al}} 2\Big(\E_{\Q}\int_0^{T\we\ta}\ff {|X(t)-Y(t)|^{\th}_{B_0}} {|X(t)-Y(t)|^{\th\ep}(h(X(t))\vee h(Y(t)))^{1-r}}\d t\Big)^{\ff 2 {\th}}\\
&\quad\times\Big(\E_{\Q}\int_0^{T\we\ta}(h(X(t))\vee h(Y(t)))^{\ff {2(1-r)} {\th-2}}\d t\Big)^{\ff {\th-2} {\th}}\\
&\leq \ff {c_0T^{\ff{\th-2} {\th}}|x-y|^{2\al}} 2\Big(\E_\Q\int_0^{T\we\ta}\ff {|X(t)-Y(t)|^{\th}_{B_0}} {|X(t)-Y(t)|^{\th\ep}(h(X(t))\vee h(Y(t)))^{1-r}}\d t\Big)^{\ff 2 {\th}}\\
&\quad\times \Big(\E_\Q\int_0^{T\we\ta}(h(X(t))\vee h(Y(t)))^{1+r}\d t\Big)^{^{\ff {2(1-r)} {\th(1+r)}}}\\
&\leq C(T,\ep,K_1,|x|,|y|)\Big(|x-y|^{2\ep}+|x-y|^{2\al}\Big)^{\ff 2 {\th}}|x-y|^{2\al}
\end{align*}
From \eqref{ForRlogR}, we have
\beg{ews*}
&\E\int_0^{T\we\ta}\ff {|X(s)-Y(s)|^{2}_{B_0}} {|X(s)-Y(s)|^{2\ep}}\d s\\
&\quad\leq T^{\ff {\th-2} {\th}}\Big(\E\int_0^{T\we\ta}\ff {|X(s)-Y(s)|^{\th}_{B_0}} {|X(s)-Y(s)|^{\th\ep}(h(X(s))\vee h(Y(s)))^{1-r}}\d s\Big)^{\ff 2 {\th}}\\
&\quad\quad\times\Big(\E\int_0^{T\we\ta}(h(X(s))\vee h(Y(s)))^{1+r}\d s\Big)^{\ff {2(1-r)} {(1+r)\th}}\\
&\quad\leq C(|x|,|y|,T,\de_2,\ep,\th)|x-y|^{\ff {4\ep} {\th}}.
\end{ews*}
Then
\beg{ews*}
\E|1-R_T|&\leq \E R_T\log R_T +2\E |\log R_T|\\
&\leq C(T,\ep,K_1,\de_2,|x|,|y|)\Big(|x-y|^{2\ep}+|x-y|^{2\al}\Big)^{\ff 2 {\th}}|x-y|^{2\al}\\
&\quad+C\Big[|x-y|^{\al+\ff {2\ep} {\th}}+|x-y|^{2\al+\ff {4\ep} {\th}}\Big],~|x-y|\leq \Big(\ff {\de_2\ep} {4le^{(\ep K_1+c(l))T}}\Big)^{\ff 1 {2\al}}.
\end{ews*}
Repeating the argument used in Theorem \ref{lmm:sFeller}, we get that there is $C>0$ depending on $|f|_\infty$, $\de_2$, $|B_0|$, $K_1$, $K_2$, $T$, $\ep$, $r$, $\th$, $|x|$, $|y|$ such that
$$|P_Tf(x)-P_Tf(y)|\leq C|x-y|^{\be},~|x-y|\leq \Big(\ff {\de_2\ep} {4le^{(\ep K_1+c(l))T}}\Big)^{\ff 1 {2\al}}$$
with $\be=(\al+\ff {2\ep} {\th})\we(\ff {2(\th+2)\al} {\th})\we (\ep-\al)$. Since $0<\al<\ep$ arbitrary, we can choose that $\be=\ff {2\th} {3\th+4}\we \ff 1 2$. Therefore, we complete the proof.
\end{proof}

\bigskip

Let $\sD(A)=\{x~|~A(x)\in \H,~x\in \V\}$. The next lemma will be used to character the domain of the non-linear operator $A$ as an operator on $\H$. The proof of the following lemma follows from \cite{Bar2010} completely. For the sake of completeness, we include the proof here.
\beg{lem}\label{dense}
Assume that (H\ref{H_hemi})-(H\ref{H_grow}). Then $\sD(A)$ is dense in $\H$.
\end{lem}
\beg{proof}
Let $\tld A(x)=\Big[1/2(K_1\vee c_3)-A\Big](x),~x\in \sD(A)$. Then, according to (H\ref{H_hemi})-(H\ref{H_coe}) and \cite[Theorem 2.4. and Corollary 2.3.]{Bar2010}, $\tld A$ is surjective and maximal monotone operator on $\H$. For $x\in \H$. Let $x_n=(I+1/n\tld A)^{-1}(x),~n\in \N$. Then
$$|x_n|^2+1/n\<\tld A(x_n),x_n\>=\<x_n,x\>\leq |x_n|\cdot|x|.$$
By (H\ref{H_coe}) and (H\ref{H_grow}), we have that
\bequ\label{bound_x_n}|x_n|^2+c_2|x_n|_\V^{r+1}/n\leq |x|\cdot|x_n|+c_1/n,\enqu
and
\bequ\label{x_x_n}
|x-x_n|_{\V^*}=|\tld A(x_n)|_{\V^*}/n\leq (K_1\vee c_3)|x_n|_{\V^*}/(2n)+c_5(|x_n|_\V^{r}+|x_n|^2+1)/n.
\enqu
Then
$$\sup_{n}\Big(|x_n|+|x_n|_\V^{r+1}/n\Big)<\infty$$
due to \eqref{bound_x_n}. So
$$\lim_{n\ra \infty}|x-x_n|_{\V^*}=0, $$
and there exists a subsequence denoted also by $\{x_n\}$ such that $x_n$ convergent in $\H$ weakly. Consequently, $x_n$ convergent to $x$ weakly in $\H$. That means for all $x\in \H$, there exists a sequence $\{x_n\}\subset \sD(A)$ that convergent to $x$ weakly. But $\overline{\sD(A)}$ is convex according to \cite[Corollary 2.5.]{Bar2010}. Therefore $x\in\overline{\sD(A)}$, i.e. $\overline{\sD(A)}=\H$.
\end{proof}

\bigskip

\noindent{\bf{\emph{Proof of Theorem \ref{thm:irreducibility}:~}}}
Let $x\in \H$, and $X(t)$ be the solution of \eqref{equ_main} starting from $x$. To prove the irreducibility of $P_T$, we have to prove that for all $y_0\in \H$ and $l>0$,
$$\P(|X(T)-y_0|\leq l)>0 \mbox{~or~} \P(|X(T)-y_0|> l)<1.$$
By Lemma \ref{dense}, we can choose $y\in \sD(A)$ such that $|y-y_0|\leq l/4$. So, we only have to prove that
$$\P\Big(|X(T)-y|> \ff {3l} 4 \Big)<1.$$
To this end, we need some preparations. The proof due to \cite{BaDa,DPZ1996,ZhangX09} essentially.

Fix $y$. Let $R>0$. Firstly, we consider the following multivalued equation
\bequ\label{multivalued}
\d z(t)=A(z(t))\d t-C_R\mathrm{sgn}(z(t)-y)\d t,~~t\in[t_1,T],
\enqu
where
$$C_R=\ff {K_1(R+|y|)} {2\Big(1-e^{-K_1(T-t_1)/2}\Big)}+|A(y)|,$$
and $\mathrm{sgn}$ is a multivalued mapping on $\H$
$$
\mathrm{sgn}(h)=\beg{cases}
~~~~\displaystyle{{h}/|h|},~~&h\neq 0,\\
~\{h~|~|h|\leq 1\},~~&h=0.
\end{cases}
$$
Due to Lemma \ref{dense}, $A$ is quasi-m-accretive on $\H$. So there is a unique mild solution to \eqref{multivalued} for arbitrary $z\in\H$ according to \cite[Corollary 4.1]{Bar2010}. We denote the solution with initial value $z$ by $z(t,z)$. Moreover, $z(\cdot,\cdot)\in C([t_1,T]\times\H;\H)$ due to \cite[Theorem 4.2]{Bar2010}. By differentiation formula for norm of solution and some necessary regularization, or see the proof of \cite[Lemma 2.3]{BaDa}, one can get that
\bequ\label{inequ:z}
|z(t,z)-y|^2\leq |z-y|^2e^{K_1(T-t_1)},~~t\in [t_1,T],
\enqu
moreover, if $|z|\leq R$, then $z(T,z)=y$. Let
$$v(t,z)=\ff {z(t,z)-y} {|z(t,z)-y|} 1_{\{z(t,z)\neq y\}}+\ff {A(y)} {C_R}1_{\{z(t,z)=y\}},~~z\in\H,$$
and $\tld Y_z(t,h)$ be the solution of the following equation with initial point $h\in \H$
$$\d \tld Y_z(t)=A(\tld Y_z(t))\d t-C_Rv(t,z)\d t,~t\in[t_1,T].$$
Then $v(t,\cdot)\in\sB_b(\H)$ for all $t\in[t_1,T]$, and for $|z|\leq R$, we have
\bequ\label{equ_y=z}
\tld Y_z(t,z)=z(t,z),~t\in[t_1,T].
\enqu
In deed, let $T_1=\inf\{t\in[t_1,T]~|~z(t,z)=y\}$. Before $T_1$, it is clear that $\tld Y_z(t,z)=z(t,z)$. Since $\tld Y_z(\cdot,z)$ and $z(\cdot,z)$ are continuous, $\tld Y_z(T_1,z)=z(T_1,z)=y$. Starting from $T_1$, $z(t,z)\equiv y$ and then $\tld Y_z(t,z)$ satisfies
$$\d \tld Y_z(t,z)=A(\tld Y_z(t,z))\d t-A(y)\d t,~\tld Y_z(T_1,z)=y,~t\geq T_1.$$
By differentiation formula and (H\ref{H_mono}), we have
$$|\tld Y_z(t,z)-y|^2\leq |y-y|^2e^{-K_1(t-T_1)}=0,~t\geq T_1.$$
So, $z(t,z)=\tld Y_z(t,z),~t\in[t_1,T]$.

Let $\ep>0$, $X_R(t_1)=X(t_1)1_{\{|X(t_1)|\leq R\}}$, and $Y^\ep$ be the solution of the following equation
$$\d Y^\ep(t)=A(Y^\ep(t))\d t-C_R(\ep B_0^{-1}+I)^{-1}v(t,X_R(t_1))\d t,\ t\geq t_1,\ Y^\ep(t_1)=X_R(t_1).$$
Since $v(t,\cdot)$ is measurable for all $t\in[t_1,T]$, it is clear that $\tld Y^\ep(t)\in \sF_{t_1}$, and adapted consequently.
By differentiation formula,
\beg{ews}\label{IforY_Y}
&\d |\tld Y_{X_R(t_1)}(t,X_R(t_1))-Y^\ep(t)|^2\\
&\quad=2_{\V^*}\Big\<A(\tld Y_{X_R(t_1)}(t,X_R(t_1)))-A(Y^\ep(t)),\tld Y_{X_R(t_1)}(t,X_R(t_1))-Y^\ep(t)\Big\>_\V\d t\\
&\quad\quad-2C_R\Big\<\Big((\ep B_0^{-1}+I)^{-1}-I\Big)v(t,X_R(t_1)),\tld Y(t)-Y^\ep(t)\Big\>\d t\\
&\quad\leq (K_1+ C_R)\Big|\tld Y_{X_R(t_1)}(t,X_R(t_1))-Y^\ep(t)\Big|^2\d t+ C_R \Big|\Big((\ep B_0^{-1}+I)^{-1}-I\Big)v(t, X_R(t_1))\Big|^2\d t.
\end{ews}
Then we obtain that for $t_1 \leq t\leq T$
\bequ\label{bounded_Y_Y_1}
|\tld Y_{X_R(t_1)}(t,X_R(t_1))-Y^\ep(t)|^2\leq C_Re^{(K_1+C_R )(T-t_1)}\int_{t_1}^{T}\Big|\Big((\ep B_0^{-1}+I)^{-1}-I\Big)v(t, X_R(t_1))\Big|^2\d t.
\enqu
Since $|(\ep B_0^{-1}+I)^{-1}-I|\leq 1$ and $|v(t,X_R(t_1))|\leq 1$, we have, in particularly at time $T$,
\bequ\label{lim}
\lim_{\ep\ra 0}\E|\tld Y_{X_R(t_1)}(T,X_R(t_1))-Y^\ep(T)|^2=\lim_{\ep\ra 0}\E|y-Y^\ep(T)|^2=0.
\enqu
On the other hand, due to $\inf_{x\in(0,T]}|\ff {1-e^{-x}} {x}|\geq e^{-T}$, we have
$$C_R(T-t_1)\leq  K_1e^T(R+|y|)+|A(y)|T.$$
Hence, it follows from \eqref{bounded_Y_Y_1} that there exists a constant  $C_{R,T,y,K_1}$ depending on $R, T,y,K_1$  such that
\bequ\label{bound_Y_Y}
\sup_{t_1\in[0,T]}\sup_{t\in[t_1,T]}|\tld Y_{X_R(t_1)}(t,X_R(t_1))-Y^\ep(t)|^2\leq C_{R,T,y,K_1}.
\enqu

Let $\tld X$ be the solution of the following equation
\beq*
\d \tld X(t)=A(\tld X(t))\d t + B(\tld X(t))\d W(t)-C_R\Big(\ep B_0^{-1}+I\Big)^{-1}v(t,X_R(t_1))1_{\{t>t_1\}}\d t,\ \tld X(0)=x.
\enq*
This equation is well define, since $\Big(C_R\Big(\ep B_0^{-1}+I\Big)^{-1}v(t,X_R(t_1))\Big)_{t\in[t_1,T]}$ is a known adapted process. By It\^o's formula, we get that
\beg{ews*}
\d |\tld X(t)-Y^\ep(t)|^2&= 2\<A(\tld X(t))-A(Y^\ep(t)),\tld X(t)-Y^\ep(t)\>\d t\\
&\quad+2\<\tld X(t)-Y^\ep(t),B(\tld X(t))\d W(t)\>+|B(\tld X(t))|^2_2\d t\\
&\leq K_1|\tld X(t)-Y^\ep(t)|^2\d t+(2|B(\tld X(t))|_2|B(Y^\ep(t))|_2+|B(Y^\ep(t))|_2^2)\d t\\
&\quad+2\<\tld X(t)-Y^\ep(t),B(\tld X(t))\d W(t)\>,\ t>t_1.
\end{ews*}
Moreover, according to \eqref{bound_Y_Y},~\eqref{inequ:z}, \eqref{equ_y=z} and \eqref{con_B_2}, there exists a constant depending on $R, T,y,K_1$ such that
$$\sup_{t\in[t_1,T]}|B(Y^\ep(t))|_2\leq C_{R,T,y,K_1}.$$
So, by Gronwall inequality,
\beg{ews*}
&\E |\tld X(T)-Y^\ep(T)|^2\\
&\quad\leq e^{K_1(T-t_1)}\Big(\E|X(t_1)-Y^\ep(t_1)|^2+C_{R,T,y,K_1}\E\int_{t_1}^{T}|B(\tld X(t))|_2\d t+C_{R,T,y,K_1}(T-t_1)\Big)\\
&\quad\leq e^{K_1T}\Big(\E\Big[|\tld X(t_1)|^21_{\{|X(t_1)|\geq R\}}\Big]+C_{R,T,y,K_1}(T-t_1)+C_{R,T,y,K_1}\E\int_{t_1}^{T}|B(\tld X(t))|_2\d t\Big).
\end{ews*}
The fist term in the last inequality goes to $0$ as $R$ tending to infinity, and the second term goes to $0$ as $t_1$ tending to $T$ if $R$ is fixed. What remains to do is to prove that the last term goes to $0$ when fixing $R$ and letting $t_1$ tend to $T$. In fact, by It\^o's formula, (H\ref{H_coe}), (H\ref{H_grow}) and
$$\Big|(\ep B_0^{-1}+I)^{-1}v(t,X_R(t_1))\Big|\leq 1,$$
we have that there exists a constant $\tld C>0$ which is independent of $t_1$ and $\ep$ such that
$$\E\sup_{t\in[0,T]}|\tld X(t)|^2+\E\int_{0}^T|\tld X(t)|_\V^{r+1}\d t<\tld C.$$
According to H\"older inequality and \eqref{bounded_B}, there exists a constant $c>0$, independent of $t_1$ and $\ep$, such that
\beg{ews*}
\E\int_{t_1}^{T}|B(\tld X(t))|_2\d t&\leq  c\E\int_{t_1}^{T}\Big(|\tld X(t)|_\V^{r+1}+|X(t)|^2+1\Big)^{1/2}\d t\\
&\leq c\sq{T-t_1}.
\end{ews*}
Therefore,
\bequ\label{last}
\E|Y^\ep(T)-\tld X(T)|^2\leq e^{K_1T}\Big(\E\Big[|\tld X(t_1)|^21_{\{|X(t_1)|\geq R\}}\Big]+C_{T,K_1,R,y}\sq{T-t_1}\Big).
\enqu

Now, we can prove that
\bequ\label{Prob_X_y}
\P\Big(|X(T)-y|> \ff {3l} 4\Big)<1.
\enqu
Due to \eqref{lim} and \eqref{last}, we choose $R$ large enough and then $\ep$ and $T-t_1$ small enough such that
$$\E|Y^\ep(T)-\tld X(T)|^2\leq \ff {l^2} {36},~\E|y-Y^\ep(T)|^2\leq \ff {l^2} {144}.$$
Let
\beg{ews*}
\tld R_T&=\exp\Big\{C_R\int_0^{T}\Big\<\tld B^{-1}(\tld X(t))(\ep B_0^{-1}+I)^{-1}v(t,X_R(t_1))1_{\{t> t_1\}},\d W(t)\Big\>\\
&\quad -\ff {C_R^2} 2 \int_0^T \Big|\tld B^{-1}(\tld X(t))(\ep B_0^{-1}+I)^{-1}v(t,X_R(t_1))1_{\{t> t_1\}}\Big|^2\d t\Big\}.
\end{ews*}
Due to (\ref{nondenerate_B}) and \eqref{rh_irr}(it is no harm to assume that $\rh\geq 1$), we obtain that
\beg{ews*}
&\E \exp\Big\{\ff {C_R^2} 2 \int_0^T \Big|\tld B^{-1}(\tld X(t))(\ep B_0^{-1}+I)^{-1}v(t,X_R(t_1))1_{\{t> t_1\}}\Big|^2\d t\Big\}\\
&\quad\leq \E \exp\Big\{ \Big(\ff { K_1(R+|y|)} {2\sq{2}(1-e^{-K_1(T-t_1)/2})}+\ff {|A(y)|} {\sq{2}}\Big)^2 \int_{t_1}^T\Big|B_0^{-1}(\ep B_0^{-1}+I)^{-1}v(t,X_R(t_1))\Big|^2\d t\Big\}\\
&\quad\leq \exp\Big\{ \ff {(R+|y|)^2 K^2_1 (T-t_1)} {4\ep^2(1-e^{-K_1(T-t_1)/2})^2} +|A(y)|^2(T-t_1) \Big\}<\infty.
\end{ews*}
Then, according to Girsanov's theorem,
$$\tld W(t):=W(t)+C_R\int_0^{t}\Big\<\tld B^{-1}(\tld X(s))(\ep B_0^{-1}+I)^{-1}v(s,X(t_1))1_{\{s> t_1\}},\d W(s)\Big\>,\ t\in[0,T]$$
is cylindrical Brownian motion under the probability measure $\tld R_T\P$ and $\tld R_T\P$ is equivalent to $\P$. Moreover, the law of $\tld X(T)$ under $\tld R_T\P$ equal to that of $X(T)$ under $\P$. Thus, under the probability measure $\P$, the law of $\tld X(T)$ is equivalent to that of $X(T)$. So, to prove (\ref{Prob_X_y}), we only have to prove that $$\P\Big(|\tld X(T)-y_0|> \ff {3l} 4\Big)<1.$$
In fact,
\beg{ews*}
&\P\Big(|\tld X(T)-y|\geq \ff {3l} 4\Big)\leq\P\Big(|\tld X(T)-Y^\ep(T)|+|Y^\ep(T)-y|\geq \ff {3l} 4\Big)\\
&\leq \P\Big(|\tld X(T)-Y^\ep(T)|+|Y^\ep(T)-y|\geq \ff {3l} 4;\ |Y^\ep(T)-y|\leq \ff l 4\Big)+\P\Big(|Y^\ep(T)-y|> \ff l 4\Big)\\
&\leq \P\Big(|X(T)-Y^\ep(T)|\geq \ff {l} {2}\Big)+\P\Big(|Y^\ep(T)-y|> l/4\Big)\leq 1/3+1/3=2/3<1.
\end{ews*}
Therefore, we have proved the irreducibility of $P_T$.
\qed

\section{Applications and examples}
In this part, we present some applications and examples.
\beg{cor}\label{cor:Holder}
Assume that $r>1$ and the conditions of Theorem \ref{thm:sFeller} hold with $\inf_{v\in\V}\rh(v)>0$ and \eqref{locally_B_Lip} replaced by
$$|(B(u)-B(v))^*(u-v)|\leq K_2|u-v|^2,~u,v\in V,$$
for some $K_2\geq 0$. Then $P_Tf$ is $\be$-H\"older continuous on $H$ for all $f\in \sB_b(H)$ with
$$0<\be<\sup_{1-\ff {(2(r-1))\we (r+1)} {\th}<\ep<1-\ff {r-1} {\th}}\Big[\ep-\inf_{p\in(0,1)}(\al_1\vee\al_2\vee\al_3)\Big],$$
where
$$\al_1=\ff {\ep\Big[2(r-1)-\th(1-\ep)\Big]} {2(p\th+1)(r-1)-\th(1-\ep)},~~\al_2=\ff {\ep(\th-2)} {2(1-p)\th+\th-2},~~\al_3=\ff {\ep-2(1-\ga)} {p\th+1}.$$
\end{cor}
\beg{proof}
Let $\ep\in (0,1)$ such that $r-1<\th(1-\ep)<2(r-1)\we(r+1)$. Consider the coupling as \eqref{coupling}. Let $2\ga=r+1-\th(1-\ep)$. Since
$$4-4\ga=4-2(r+1-\th(1-\ep))=2\th(1-\ep)-2(r-1)<\th(1-\ep),$$
there is $C>0$ depending on $K_1,~K_2,~\de_1,~\ep,~\th,~r,~T,~|B_0|$ such that
\beg{align*}
&\E\exp\Big\{u\int_0^{T\we\ta}e^{-(1-\ga)K_1t}\ff {|X(t)-Y(t)|_{B_0}^{\th}} {|X(t)-Y(t)|^{\th\ep}}\d t-\ff {u|x-y|^{2-2\ga}} {\de_1(1-\ga)}\Big\}\\
&\quad\leq \Big(\E\exp\Big\{\ff {u^2K_2^2} {\de_1^2} \int_0^{T\we\ta}e^{-2(1-\ga)K_1t}|X(t)-y(t)|^{4-4\ga}\d t\Big\}\Big)^{1/2}\\
&\quad\leq \Big(\E\exp\Big\{\ff u {|B_0|^\th}\int_0^{T\we\ta}e^{-(1-\ga)K_1t}|X(t)-Y(t)|^{\th(1-\ep)}\d t+Cu^{[2-\ff {4-4\ga} {\th(1-\ep)}]/[1-\ff {4-4\ga} {\th(1-\ep)}]}\Big\}\Big)^{1/2}\\
&\quad\leq\Big(\E\exp\Big\{u\int_0^{T\we\ta}e^{-(1-\ga)K_1t}\ff {|X(t)-Y(t)|_{B_0}^{\th}} {|X(t)-Y(t)|^{\th\ep}}\d t+Cu^{[2-\ff {4-4\ga} {\th(1-\ep)}]/[1-\ff {4-4\ga} {\th(1-\ep)}]}\Big\}\Big)^{1/2},
\end{align*}
according to (\ref{HofA}), (\ref{ItoforX_Y}), \eqref{Ito_ga} and H\"older inequality.
So there is $C>0$ depending on $K_1$, $K_2$, $\de_1$, $\ep$, $\th$, $r$, $T$, $|B_0|$ such that
\beg{ews*}
&\E\exp\Big\{u\int_0^{s\we\ta}e^{-(1-\ga)K_1t}\ff {|X(t)-Y(t)|_{B_0}^\th} {|X(t)-Y(t)|^{\th\ep}}\d t\Big\}\\
&\quad\leq\exp\Big\{C\Big(u^{\ff {2(r-1)} {2(r-1)-\th(1-\ep)}} +u|x-y|^{2-2\ga}\Big)\Big\}<\infty,~~u\in\R^+.
\end{ews*}
Defining $R_T$ as in lemma \ref{lmm:sFeller}, we obtain that there exists $C$  depending on $K_1$, $K_2$, $\de_1$, $\ep$, $\th$, $r$, $T$, $|B_0|$ such that
\beg{align}\label{exp}
\E R_T^2&\leq \Big(\E\exp\Big\{\ff {7|x-y|^{2\al}} 2\int_0^{T\we\ta}\ff {|\tld B^{-1}(Y(t))(X(t)-Y(t))|^2} {|X(t)-Y(t)|^{2\ep}}\d t\Big\}\Big)^{1/2}\nonumber\\
&\leq \Big(\E\exp\Big\{\ff {7|x-y|^{2\al}} 2\int_0^{T\we\ta}\ff {|(X(t)-Y(t))|_{B_0}^2} {|X(t)-Y(t)|^{2\ep}}\d t\Big\}\Big)^{1/2}\nonumber\\
&\leq \Big(\E\exp\Big\{C\Big(|x-y|^{p\al\th}\int_0^{T\we\ta}\ff {|X(t)-Y(t)|_{B_0}^\th} {|X(t)-Y(t)|^{\th\ep}}\d t+ (\th-2)|x-y|^{\ff {2\al\th(1-p)} {\th-2}}\Big)\Big\}\Big)^{1/2}\\
&\leq\exp\Big\{C\Big(|x-y|^{\ff {2p\th\al(r-1)} {2(r-1)-\th(1-\ep)}}+({\th-2})|x-y|^{\ff {2\al\th(1-p)} {\th-2}}+|x-y|^{2(1-\ga)+p\th\al}\Big)\Big\}\nonumber\\
&=:\exp\{\Phi(|x-y|)\}<\infty.\nonumber
\end{align}
with $p\in(0,1)$ and if $\th=2$ we set $({\th-2})|x-y|^{\ff {2\al\th(1-p)} {\th-2}}=0$. It is clear that $\Ph$ is continuous and $\Ph(0)=0$. By Girsanov's theorem, $\{\tld W(t)\}_{t\geq0}$ is cylindrical Wiener process under the probability measure $R_T\P$. Repeating the same argument in Lemma \ref{lmm:sFeller} and via the following inequality
$$\E|1-R_T|\leq \Big(\E|1-R_T|^2\Big)^{1/2}=\Big(\E R_T^2-1\Big)^{1/2}\leq \Ph^{1/2}(|x-y|)\exp\{\Ph(|x-y|)/2\},$$
we obtain that
$$|P_Tf(x)-P_Tf(y)|\leq |f|_\infty\Big[\Ph^{1/2}(|x-y|)\exp\{\Ph(|x-y|)/2\}+|x-y|^{\ep-\al}\ff {e^{\ff {\ep K_1T} {2}}-1+T^2} {T^2}\Big].$$
So
$$|P_Tf(x)-P_Tf(y)|\leq C|x-y|^{\be}$$
with
$$\be=\ff {2p\th\al(r-1)} {2(r-1)-\th(1-\ep)}\we \ff {2\al\th(1-p)} {\th-2}\we (2(1-\ga)+p\th\al)\we (\ep-\al).$$
Since $p\in(0,1)$ and $\ep$ is arbitrary such that $(r-1)<\th(1-\ep)<(2(r-1))\we (r+1)$,
$$\be<\sup_{1-\ff {(2(r-1))\we (r+1)} {\th}<\ep<1-\ff {r-1} {\th}}\Big[\ep-\inf_{p\in(0,1)}(\al_1\vee\al_2\vee\al_3)\Big].$$
Therefore, the proof is completed.
\end{proof}
\beg{rem}
Since the coupling used here neither is optimal nor succeed in deterministic finite time, the H\"older continuity got in Lemma \ref{lmm:sFeller}, Lemma \ref{lmm:sFeller'} and the corollary above are worse comparing with results obtained in \cite{WBook13',Wang14b}.
\end{rem}

Combining Theorem \ref{thm:sFeller}, Theorem \ref{thm:sFeller'} and Theorem \ref{thm:irreducibility}, we have
\beg{cor}
The same conditions of Theorem \ref{thm:sFeller}(or Theorem \ref{thm:sFeller'}) with $\inf_{v\in\V}\rh(v)>0$ and \eqref{con_B_2} hold. Moreover, we assume that $|\cdot|_{B_0}$ is bounded on bounded sets of $\V$. Then $P_t$ has a unique invariant measure $\nu$ with full support on $H$, and for all probability measure $\mu$ on $\H$
$$\lim_{t\ra\infty} ||P^*_t\mu-\nu||_{VT}=0,$$
where $||\cdot||_{VT}$ is the total variation norm and $P_t^*$ is the adjoint operator of $P_t$.
\end{cor}
\beg{proof}
It is a direct consequence of the Doob theorem and \cite[Corollary 1.]{Stet94}.
\end{proof}

We present some specific examples to illustrate our main results. These examples have been intensively study in additive noise case. Here we do not repeat the details, one can consult \cite{Wang2007,LiuW08,Liu09,WBook13'} for more general discussion.  The non-degenerate condition \eqref{nondenerate_B} allows us to compare the multiplicative noise with the additive noise situation.

\beg{exa}\label{exa_1}
(Stochastic porous media equations)
Let $D\subset \R^d$ be a bounded domain with smooth boundary, $\mu$ be the normalized Lebesgue measure on $D$. Let $\De_D$ be the Dirichlet Laplace on $D$ and $L=-(-\De_D)^\ga$ for some constant $\ga>0$. Consider the Gel'fand triple $$L^{r+1}(D,\mu)\subset H^\ga(D,\mu)\subset (L^{r+1}(D,\mu))^*,$$
where $H^{\ga}(D,\mu)$ is the completion of $L^2(D,\mu)$ under the norm
\bequ\label{norm}
|x|:=\Big(\int_{D}|(-\De_D)^{-\ga/2}x|^2\d\mu\Big)^{1/2},~x\in L^2(D,\mu),
\enqu
and $(L^{r+1}(D))^*$ is the dual space of $L^{r+1}(D,\mu)$ w.r.t. $H^\ga(D,\mu)$. Let $$\la_1\leq \la_2\leq \cdots\leq \la_n\leq \cdots$$
be the eigenvalues of $-\De_D$ including multiplicities with unite eigenfunctions $\{e_j\}_{j\geq 1}$. Let $r>1$,
$$\Psi(s)=s|s|^{r-1},~\Ph(s)=cs,~B_0e_j=j^{-q}e_j,~j\geq 1,$$
for some constants $c\geq 0$ and $q>1/2$. Then $B_0\in \sL_2(H^\ga(D,\mu))$. Let $$B(x)e_j=b_j(x)j^{-q}e_j,~j\geq 1,$$
with $b\in \R$ such that
\beg{ews}\label{B_exa}
|b_j(x)-b_j(y)|\leq b|x-y|,&~x,~y\in H^{\ga}(D,\mu),\\
\inf_{|x|\leq R}\inf_{j\geq 1} b_j(x)>0,&~R>0.
\end{ews}
We consider the following equation under the triple introduced above
\bequ\label{exa_porous/fast}
\d X(t)=\Big(L\Ps(X(t))+\Ph(X(t))\Big)\d t+B(X(t))\d W(t),
\enqu
where $W(t)$ is cylindrical Brownian motion on $\H$ w.r.t. a complete filtered probability space $(\Om, \{\sF\}_{t\geq 0},\P)$. According to \cite[Corollary 3.1]{Wang2007}, or \cite[Example 2.4.1]{WBook13'}, if $\ga\geq dq$, then \eqref{HofA}
holds for all $\th\in (r-1,r+1]$. By \eqref{B_exa} and the definition of the noise term $B$, \eqref{nondenerate_B} holds with $\rh(x)=\inf_{j\geq 1} b_j(x)$. So the claim in Theorem \ref{thm:sFeller} holds. If $\inf_{v, j\geq1}b_j(x)>0$, then Theorem \ref{thm:irreducibility} holds.
\end{exa}

\beg{exa}
(Stochastic fast diffusion equations) Let $D=(0,1)\subset \R^1$, $\mu$, $H^\ga(D,\mu)$, $L$, $\Ph$, $\Ps$ as in Example \ref{exa_1}. Let $\ff 1 3<r<1$, $\ga=1$ $\V=L^{r+1}(D,\mu)\bigcap H^{1}(D,\mu)$ with
$$
|v|_{\V}=|v|_{L^{r+1}}+|v|,~v\in \V.$$
We consider the equation \eqref{exa_porous/fast} under the triple
$$\V\subset H^{1}(D,\mu)\subset \V^*.$$
Let $B_0$ and $B$ defined as in Example \ref{exa_1} with $q$ to be determined later. According to \cite{LiuW08}, or \cite[Example 2.4.2]{WBook13'}, for all $\th\in (\ff 4 {r+1}, \ff {6r+2} {r+1})$ and $q\in (\ff 1 2,\ff {3r+1} {\th(r+1)})$ such that \eqref{HofA'} and \eqref{H_coe'} hold with $h(v)=|v|_{\V},~v\in\V$. Then Theorem \ref{thm:sFeller'} holds.
\end{exa}
At last, we give an explicit example of $B$ satisfying \eqref{B_exa}.
\beg{exa}
Let $D\subset\R^d$, $\mu$, $\{\la_j\}_{j\geq 1}$, $\{e_j\}_{j\geq 1}$ are defined as in Example \ref{exa_1}. Let
$$b_j(x)= \ff 1 {1+j^{-\ff {2\ga} {d}}|\mu(xe_j)|},~x\in L^{r+1}(D,\mu),~j\geq 1,$$
where $\mu(x):=\int_{D}x\d \mu$. By the Weyl's formula, see \cite{Weyl}, or \cite{Gull}, there is a constant $C_d$ depending on $d$ such that $\lim_{j\ra \infty}\ff {j^{2/d}} {\la_j}=C_d \mathbf{Vol}(D)$, where $\mathbf{Vol}(D)$ means the volume of $D$. So there is $C$ independent of $j$ such that
$$\ff 1 C |\<x,e_j\>|\leq j^{-\ff {2\ga} {d}}|\mu(xe_j)|\leq C|\<x,e_j\>|,~x\in L^{r+1}(D,\mu),~j\geq 1,$$
where $\<\cdot,\cdot\>$ is the inner product induced by the norm \eqref{norm}. It is easy to check that \eqref{B_exa} holds.
\end{exa}

\bigskip

\noindent\textbf{Acknowledgements}

\medskip

The author would like to thank Professor Feng-Yu Wang for his useful suggestions.

\end{document}